\documentclass{amsart}
\usepackage{verbatim}
\usepackage{rotating} 
\usepackage{amssymb}
\usepackage[bookmarks,colorlinks,breaklinks]{hyperref}  
\hypersetup{linkcolor=blue,citecolor=blue,filecolor=blue,urlcolor=blue}
\usepackage{mathrsfs,amsfonts,amsmath,amssymb,epsfig,amscd,xy,amsthm} 

\newtheorem{theorem}{Theorem}[section]
\newtheorem{proposition}[theorem]{Proposition}

\newtheorem{lemma}[theorem]{Lemma}

\newtheorem{question}[theorem]{Question}

\theoremstyle{plain}

\theoremstyle{remark}

\newtheorem{remark}[theorem]{Remark}
\newtheorem{example}[theorem]{Example}

\newcommand{\C}{{\mathbb C}}
\newcommand{\F}{{\mathbb F}}
\newcommand{\Q}{{\mathbb Q}}

\newcommand{\Z}{{\mathbb Z}}
\newcommand{\N}{{\mathbb N}}

\newcommand{\cX}{{\mathcal X}}

\newcommand{\cI}{{\mathcal I}}
\newcommand{\cJ}{{\mathcal J}}

\newcommand{\Gmn}{\mathbb{G}_{\text{m}}^{n}}

\DeclareMathOperator{\GL}{GL}

\DeclareMathOperator{\End}{End}

\DeclareMathOperator{\id}{id}

\newcommand{\bP}{{\mathbb P}}

\newcommand{\bG}{{\mathbb G}}
\newcommand{\bfa}{{\mathbf a}}
\newcommand{\bfp}{{\mathbf p}}

\newcommand{\bfu}{{\mathbf u}}
\newcommand{\bfv}{{\mathbf v}}
\newcommand{\bfw}{{\mathbf w}}
\newcommand{\bfx}{{\mathbf x}}
\newcommand{\bfy}{{\mathbf y}}

\newcommand{\bC}{{\mathbb C}}

\newcommand{\bA}{{\mathbb A}}

\newcommand{\lra}{\longrightarrow}
\newcommand{\cO}{\mathcal{O}}
\newcommand{\cB}{\mathcal{B}}
\newcommand{\cF}{\mathcal{F}}

\newcommand{\cU}{\mathcal{U}}

\newcommand{\Cp}{\bC_p}

\author{D.~Ghioca}
\address{
Dragos Ghioca\\
Department of Mathematics\\
University of British Columbia\\
Vancouver, BC V6T 1Z2\\
Canada
}
\email{dghioca@math.ubc.ca}

\author{K.~D.~Nguyen}
\address{
Khoa D.~Nguyen \\
Department of Mathematics\\
University of British Columbia\\
And Pacific Institute for The Mathematical Sciences\\ 
Vancouver, BC V6T 1Z2, Canada}
\email{dknguyen@math.ubc.ca}
\urladdr{www.math.ubc.ca/\~{}dknguyen}

\keywords{Dynamical Mordell-Lang problem, intersections of orbits, affine transformations, semiabelian varieties}
\subjclass[2010]{Primary: 37P55, 11C99. Secondary: 11B37, 11D61.}

\begin{document}
	\title[The orbit intersection problem]{The orbit intersection problem for linear spaces and semiabelian varieties}

	\begin{abstract}
	Let $f_1,f_2:\C^N\lra \C^N$ be affine maps $f_i(\bfx):=A_i\bfx + \bfy_i$ (where each $A_i$ is an $N$-by-$N$ matrix and $\bfy_i\in \C^N$), and let $\bfx_1,\bfx_2\in \bA^N(\C)$ such that $\bfx_i$ is not preperiodic under the action of $f_i$ for $i=1,2$. If none of the eigenvalues of the matrices $A_i$ is a root of unity, then we prove that the set $\{(n_1,n_2)\in \N_0^2\colon f_1^{n_1}(\bfx_1)=f_2^{n_2}(\bfx_2)\}$ is a finite union of sets of the form $\{(m_1k+\ell_1,m_2k+\ell_2)\colon k\in\N_0\}$ where $m_1,m_2,\ell_1,\ell_2\in\N_0$. Using this result, we prove that for any two self-maps $\Phi_i(x):=\Phi_{i,0}(x)+y_i$ on a semiabelian variety $X$ defined over $\C$ (where $\Phi_{i,0}\in \End(X)$ and $y_i\in X(\C)$), if none of the eigenvalues of the induced linear action $D\Phi_{i,0}$ on the tangent space at $0\in X$ is a root of unity (for $i=1,2$), then for any two non-preperiodic points $x_1,x_2$, the set $\{(n_1,n_2)\in \N_0^2\colon \Phi_1^{n_1}(x_1)=\Phi_2^{n_2}(x_2)\}$ is a finite union of sets of the form $\{(m_1k+\ell_1,m_2k+\ell_2)\colon k\in\N_0\}$ where $m_1,m_2,\ell_1,\ell_2\in\N_0$. We give examples to show that the above condition on eigenvalues is necessary and introduce certain geometric properties that imply such
	a condition. Our method involves an analysis of certain systems of polynomial-exponential equations and the $p$-adic exponential map for semiabelian varieties.  
	\end{abstract}

	\maketitle{}


 \section{Introduction}\label{sec:introduction}
Throughout this paper, let $\N$ denote the set of positive integers,
$\N_0:=\N\cup\{0\}$, and let $K$ be an algebraically closed field of characteristic $0$.
An arithmetic progression is a set of the form
$\{mk+\ell:\ k\in\N_0\}$  
for some $m,\ell\in\N_0$; note that when $m=0$, this set
is a singleton. 
For a map $f$ from a set $X$ to itself and for $m\in\N$, we let $f^m$
denote the $m$-fold iterate $f\circ\ldots\circ f$, and let $f^0$
denote the identity map on $X$. If $x\in X$, we define the orbit
$\cO_f(x):=\{f^n(x)\colon n\in\N_0\}$.  We say that $x$ is preperiodic (or more precisely $f$-preperiodic) if its orbit $\cO_f(x)$ is finite. 

In algebraic dynamics, one studies the system 
$\{\Phi^n:\ n\in\N_0\}$ when $X$ is a (quasi-projective) variety over
$K$ and $\Phi$ is a $K$-morphism. Motivated by the classical Mordell-Lang conjecture proved by Faltings \cite{Fal94} and Vojta \cite{Voj96}, the Dynamical Mordell-Lang Conjecture predicts that for a given
$x\in X(K)$ and a closed subvariety $V$ of $X$, the set $\{n\in\N\colon \Phi^n(x)\in V(K)\}$
is a finite union of arithmetic progressions (see \cite[Conjecture~1.7]{GT-JNT} along with the earlier work  of Denis \cite{Den94} and Bell \cite{Bel06}). Considering $X$ a semiabelian variety and $\Phi$ the translation by a point $x\in X(K)$, one recovers the cyclic case in the classical Mordell-Lang conjecture from the above stated Dynamical Mordell-Lang Conjecture. So, it is natural to seek a generalization of the Dynamical Mordell-Lang Conjecture (see \cite{GTZ-Bordeaux}) to a statement which would contain as a special case the full statement of the classical  Mordell-Lang conjecture. Therefore one can study the general \emph{dynamical Mordell-Lang problem} by considering commuting $K$-morphisms
$f_1,\ldots,f_r$ from $X$ to itself and 
asking whether the set 
$$S(X,f_1,\ldots,f_r,x,V):=\{(n_1,\ldots,n_r)\in\N_0^r\colon f_1^{n_1}\circ\cdots \circ f_r^{n_r}(x)\in V(K)\}$$ is a finite union of translates of subsemigroups of 
$\N_0^r$. This more general problem turns out to be rather
delicate even when the underlying variety $X$ is an algebraic group
and each $f_i$ is an endomorphism. Indeed, a recent theorem of
Scanlon-Yasufuku \cite{SY2013} establishes that for any system of polynomial-exponential equations, its set of solutions is equal to
a set of the form $S(X,f_1,\ldots,f_r,x,V)$
where $X$ is an algebraic torus, $V$ is an algebraic subgroup, and $f_1,\ldots,f_r$ are some commuting endomorphisms of $X$.

The Dynamical Mordell-Lang Conjecture has sparked
significant interest and there have been many partial results; we refer the readers to \cite{book} for a survey of recent work. On
the other hand, there are very few results known for the more general
 dynamical Mordell-Lang problem. In fact, not much is known even when we restrict to the following special case called \emph{the orbit intersection problem}:
\begin{question}
\label{main question}
Let $X$ be a variety over a $K$, let $r\geq 2$. For $1\leq i\leq r$, let $\Phi_i$ be a $K$-morphism from $X$ to itself, and let $\alpha_i\in X(K)$ that is not $\Phi_i$-preperiodic. 
When can we conclude that the set $S:=\{(n_1,\ldots,n_r)\in \N_0^r\colon \Phi_1^{n_1}(\alpha_1)=\ldots = \Phi_{r}^{n_r}(\alpha_r)\}$ is a finite union of sets of the form $\{(n_1k+\ell_1,\ldots,n_rk+\ell_r)\colon k\in\N_0\}$
for some $n_1,\ldots,n_r,\ell_1,\ldots,\ell_r\in\N_0$?
\end{question}  

\begin{remark}
For $1\leq i\leq r$, let $f_i$ be the self-map of $X^r$
induced by the map $\Phi_i$ on the $i$-th factor and the identity
map on all the other factors.
Let $\Delta$ be the diagonal of 
$X^r$. The set $S$ in Question~\ref{main question} is
exactly the set of $(n_1,\ldots,n_r)\in\N_0^r$ such that 
$(f_1^{n_1}\circ \ldots\circ f_r^{n_r}) (\alpha_1,\ldots,\alpha_r)\in \Delta$. This explains why Question~\ref{main question} is a special case of the general dynamical Mordell-Lang problem with the further requirement that $S$ is a finite union of translates of subsemigroups of $\N_0^r$ \emph{whose rank is at most $1$}. When some $\alpha_i$
is preperiodic, it is trivial to describe the set $S$. This justifies
our assumption on the $\alpha_i$'s.
\end{remark}

\begin{remark}\label{rem:no automorphism}
Motivated by the examples in \cite[Section~6]{GTZ-Bordeaux}, one may ask whether
the following condition is sufficient for Question~\ref{main question}: there do not exist $m\in \N$ and a positive dimensional
closed subvariety $Y$ of $X$ such that $\Phi_i^m$ 
restricts to an automorphism on $Y$ for some $i\in\{1,\ldots,r\}$. 
This condition is indeed sufficient when $X$ is a 
semi-abelian variety (see Theorem~\ref{main result}
and Proposition~\ref{prop:no automorphism}); also this condition is often necessary as shown by various examples such as the following one. If $X=\bA^2$, $m=2$, and  $\Phi_1:\bA^2\lra \bA^2$ is given by $\Phi_1(x,y)=(x+y^2,y^3)$ while $\Phi_2:\bA^2\lra \bA^2$ is given by $\Phi_2(x,y)=(x^2+y,y^2)$, then both $\Phi_1$ and $\Phi_2$ restrict to endomorphisms of $V:=\bA^1\times\{1\}$, and moreover $(\Phi_1)|_V$ is actually an automorphism. Then the set of all pairs $(n_1,n_2)$ such that $\Phi_1^{n_1}(0,1)=\Phi_2^{n_2}(0,1)$ is infinite, but it is not a finite union of cosets of subsemigroups of $\N_0\times \N_0$.
\end{remark}

Question~\ref{main question} in the case $X=\bP^1_K$ and 
each $\Phi_i$ is a polynomial of degree larger than $1$ has been
settled by Tucker, Zieve, and the first author \cite{Inventiones,GTZ-Duke}. They also obtain
various results for the general Mordell-Lang problem when
$X$ is a semiabelian variety and the self-maps
are endormophisms satisfying certain technical conditions
\cite{GTZ-Bordeaux}. The case when $X=\bP^1_K$ endowed by the action of 
certain generic rational functions is also established in an
ongoing joint work of Zieve and the second author.

\emph{The goal of this paper is to answer Question~\ref{main question} when $X$ is a semiabelian variety and when $X=\bA^n_K$
and the self-maps are affine transformations.}

\begin{theorem}
\label{main result}
Let $X$ be a semiabelian variety over $K$ and $r\geq 2$. For
$1\leq i\leq r$, let $\Phi_i:\ X\rightarrow X$ be a $K$-morphism and let $\alpha_i\in X(K)$ that is
not $\Phi_i$-preperiodic. Let $\Phi_{i,0}$ be 
a $K$-endomorphism of $X$ and $\alpha_{i,0}\in X(K)$
such that $\Phi_i(x)=\Phi_{i,0}(x)+\alpha_{i,0}$
and let $D\Phi_{i,0}$ be the linear transformation of the tangent space at
the identity of $X$ induced by $\Phi_{i,0}$. If
none of the eigenvalues of $D\Phi_{i,0}$ is a root of unity
for every $i$, then the set
$$S:=\{(n_1,\ldots,n_r)\in \N_0^r\colon \Phi_1^{n_1}(\alpha_1)=\ldots=\Phi_r^{n_r}(\alpha_r)\}$$
is a finite union of sets of the form
$\{(n_1k+\ell_1,\ldots,n_rk+\ell_r)\colon k\in\N_0\}$
for some $n_1,\ldots,n_r,\ell_1,\ldots,\ell_r\in\N_0$.
\end{theorem}

We note that any self-map of a semiabelian variety is indeed a composition of a translation with an algebraic group endomorphism (see \cite[Theorem~5.1.37]{Noguchi}). The structure for self-maps on semiabelian varieties $X$ is similar to the structure of affine self-maps on $\bA^N$, and this allows us to reduce Theorem~\ref{main result} (using the $p$-adic exponential map on $X$) to proving Question~\ref{main question} for affine endomorphisms of $\bA^N$ (see Theorem~\ref{main result linear algebra}).

\begin{example}
We present an example to illustrate that the conclusion of Theorem~\ref{main result} would fail without the assumption on the eigenvalues of
the linear maps $D\Phi_{i,0}$'s. Consider the case $X=\bG_m$, $\Phi_1(x)=2x$, 
$\Phi_2(x)=x^2$, $\alpha_1=1$, and $\alpha_2=2$; then $\{(m,n)\in\N_0\times \N_0\colon \Phi_1^m(\alpha_1)=\Phi_2^n(\alpha_2)\}=\{(2^n,n)\colon n\in \N_0\}$.
\end{example}

In Section~\ref{section proof}, we present the proof of Theorem~\ref{main result} and give some geometric conditions
that imply the condition on the linear transformations $D\Phi_{i,0}$
in Theorem~\ref{main result}. For instance, let $\Phi_0$ be an endomorphism
of a semiabelian variety $X$ defined over $K$, then  
none of the eigenvalues of $D\Phi_0$ is a root of unity 
\emph{if and only if} 
$\Phi_0$
does not preserve a non-constant fibration
 (see Proposition~\ref{prop:no fibration}).
Here we say that $\Phi_0$ preserves a non-constant fibration if there exists a non-constant \emph{rational} map
$f:\ X\rightarrow \bP^1_K$ such that $f\circ \Phi_0 = f$.

In \cite[Theorem~1.3 (a)]{GTZ-Bordeaux},  a special case of Theorem~\ref{main result} was obtained, i.e. when each $\Phi_i=\Phi_{i,0}$ is a group endomorphism and moreover, the Jacobians at $0\in X$ of each $\Phi_i$ is \emph{diagonalizable}. The hypothesis from \cite{GTZ-Bordeaux} about the diagonalizability of the  Jacobians of $\Phi_i$  
greatly simplifies the problem since it
allows one to reduce the problem to classical unit equations in diophantine geometry. In the absence of the diagonalizability condition, we have to use a much more refined analysis of the pairs $(m,n)\in \N_0\times \N_0$ such that $a_m=b_n$ for two arbitrary linear recurrence sequences. The result from \cite{GTZ-Bordeaux} dealt only with the much easier case when the characteristic polynomials for these two linear recurrence sequences have \emph{non-repeated} roots. As it was noted in \cite[Section~6]{GTZ-Bordeaux}, if one of the maps $\Phi_i$  is an automorphism of $X$ (or induces an automorphism of a positive dimensional subvariety of $X$), then the set $S$ may no longer be a finite union of cosets of subsemigroups of $\N_0^r$. Essentially, the problem with one of the endomorphisms being actually an automorphism is the following: assuming $X$, $\alpha_i$ and $\Phi_i$ are defined over a number field, then the points in $\cO_{\Phi_i}(\alpha_i)$ are not sufficiently sparse with respect to a Weil height on $X$ and this increases the probability that  $\cO_{\Phi_i}(\alpha_i)$ intersects the other orbits.

The most important ingredient in the proof of Theorem~\ref{main result}
is the following result which also answers Question~\ref{main question}
when $X=\bA^n_K$ and the maps $\Phi_i$'s are affine transformations.

\begin{theorem}
\label{main result linear algebra}
Let $r,N\in\N$ with $r\geq 2$. For $i\in\{1,\ldots,r\}$, let 
$f_i:\ K^N\lra K^N$ be an affine map which means there exist 
an $N\times N$-matrix $A_i\in M_N(K)$ and a vector $\bfx_i\in K^N$ such that $f(\bfx)=A_i\bfx+\bfx_i$ for every $\bfx\in K^N$.
 For $i\in\{1,\ldots,r\}$, let $\bfp_i\in K^N$ that is not
 $f_i$-preperiodic.
 If none of the eigenvalues of $A_i$ is a root of unity for each $i\in\{1,\ldots,r\}$, then the set $$S:=\{(n_1,\ldots,n_r)\in \N_0^r\colon f_1^{n_1}(\bfp_1)=\ldots=f_r^{n_r}(\bfp_r)\}$$ is a finite union of sets of the form
 $\{(n_1k+\ell_1,\ldots,n_rk+\ell_r):\ k\in \N_0\}$ for
 some $n_1,\ldots,n_r,\ell_1,\ldots,\ell_r\in\N_0$.
\end{theorem}
The conclusion of this theorem would fail without the
assumption on the eigenvalues of the matrices $A_i$. For example,
let $N=1$, $r=2$, $A_1(x)=x+1$, $A_2(x)=2x$, $\bfp_1=0$, and $\bfp_2=1$,
then $S=\{(2^n,n):\ n\in\N_0\}$. Also, Theorem~\ref{main result linear algebra} fails if one does not assume the maps $f_i$ are affine, as shown by the following example. Let $r=2$, $n=1$, $f_1(x)=2x$, $f_2(x)=x^2$, $\bfp_1=1$ and $\bfp_2=2$; then $S=\{(2^n,n)\colon n\in \N_0\}$.

The organization of this paper is as follows. In Section~\ref{sec:linear algebra}, we present the proof of Theorem~\ref{main result linear algebra} which requires
a careful analysis of a certain system of polynomial-exponential equations in two variables. Some results on polynomial-exponential
equations are given in the next section following Schmidt's exposition
\cite{Sch03}. In Section~\ref{section proof}, Theorem~\ref{main result} is reduced to Theorem~\ref{main result linear algebra} thanks to the use of the $p$-adic exponential
map for an appropriate choice of the prime $p$.

We conclude this section with a brief discussion of the dynamical Mordell-Lang problem over fields of positive characteristic. We note right from the start that Question~\ref{main question} fails even in the simplest examples of affine maps defined over $\F_p(t)$. Indeed, let  $\Phi_i:\bA^1\lra \bA^1$ be affine maps given by $\Phi_1(x)=tx-t+1$ and $\Phi_2(x)=(t+1)x$. It is immediate to see that $\Phi_1^m(2)=t^m+1$ while $\Phi_2^n(1)=(t+1)^n$. Then the set $$S=\{(m,n)\in\N_0^2\colon \Phi_1^m(2)=\Phi_2^n(1)\}=\{(p^n,p^n)\colon n\in\N_0\}.$$
The above example stems from similar examples disproving a naive formulation of the Dynamical Mordell-Lang Conjecture in positive characteristic. 
A variant of the Dynamical Mordell-Lang Conjecture has been proposed by Scanlon and the first author
\cite[Chapter~13]{book}; however there are very few partial results
since even the case of $\Gmn$ seems to be closely related to
very difficult problems in diophantine geometry. We refer
the readers to the discussion in \cite[Section~13.3]{book}
 for more details. A deep theorem of Adamczewski and
 Bell \cite[Theorem~1.4]{AB12}
 implies that if $K$ is a field of characteristic $p>0$, then
 the set $S$ in Theorem~\ref{main result linear algebra} is
  $p$-automatic.

\medskip

{\bf Acknowledgments.} The first author is partially supported by NSERC and the second author is partially supported
	by a UBC-PIMS postdoctoral fellowship.


\section{Some diophantine equations involving linear recurrence sequences}\label{sec:diophantine}
\subsection{Some classical results} A large part of this
subsection follows the notation from Schmidt's article \cite{Sch03}.
All the sequences considered in this section are sequences
of complex numbers.
A tuple $(a_1,\ldots,a_k)$ of non-zero numbers is called
non-degenerate if $\displaystyle\frac{a_i}{a_j}$ is not a root of
unity for $1\leq i< j\leq k$. A linear recurrence sequence is called non-degenerate if the tuple of (non-zero) characteristic roots
is non-degenerate.
We begin with the following well-known
result:
\begin{theorem}[Skolem-Mahler-Lech]\label{thm:SML}
Let $\{u_n:\ n\in\N_0\}$ be a linear recurrence sequence. Then the set $Z:=\{n:\ u_n=0\}$ is
a finite union of arithmetic progressions. Furthermore, if
$u_n$ is non-degenerate then $Z$ is finite.
\end{theorem}

We now consider non-degenerate linear recurrence sequences
that are \emph{not} of the form $P(n)\alpha^n$ where $\alpha$ is a
root of unity. It is convenient to write such
a sequence $u$ as:
\begin{equation}\label{eq:convention u}
u_n=\sum_{i=0}^q P_i(n)\alpha_i^n
\end{equation}
\emph{with the following convention} \cite[Section~11]{Sch03}. If some root of the characteristic
polynomial is a root of unity, let this root be
$\alpha_0$, and $\alpha_1,\ldots,\alpha_q$ the other roots. If no
root of the characteristic polynomial is
a root of unity, let these roots
be $\alpha_1,\ldots,\alpha_q$, and set $\alpha_0=1$, $P_0=0$.
Let $v$ be another sequence written as
\begin{equation}\label{eq:convention v}
v_n=\sum_{i=0}^{q'}Q_i(n)\beta_i^n
\end{equation}
with the same convention. The two sequences $u$ and $v$ are said to be
\emph{related} if $q=q'$ and after a suitable reordering of 
$\beta_1,\ldots,\beta_q$ we have:
$$\alpha_i^a=\beta_i^b\ \text{for every $i\in\{1,\ldots,q\}$}$$
for certain non-zero integers $a$ and $b$.

The next result follows from Schmidt's reformulation
of a theorem by Laurent whose
proof uses the celebrated Subspace Theorem:
\begin{theorem}[Laurent]\label{thm:Schmidt reformulation}
Let $u$ and $v$ be non-degenerate linear recurrence sequences
given by \eqref{eq:convention u}, \eqref{eq:convention v}, and 
under the convention described above. Consider the equation:
$$u_m=v_n\ \text{for $(m,n)\in\N_0^2$}$$
and let $Z$ be the set of solutions.
We have the following:
\begin{itemize}
\item [(a)] If $u$ and $v$ are not related then $Z$ is finite.
\item [(b)] $P_0(m)\alpha_0^m=Q_0(n)\beta_0^n$ for all but finitely many
$(m,n)\in Z$.
\end{itemize} 
\end{theorem} 
\begin{proof}
This follows from \cite[Theorem~11.2]{Sch03}.
\end{proof}

\subsection{Some consequences}
\begin{proposition}\label{prop:Schmidt note}
Let $k\in\N$, let $a,b_1,\ldots,b_k\in \C^*$ none of which
is a root of unity. Let $P(x),Q_1(x),\ldots,Q_k(x)\in\C[x]\setminus\{0\}$ and let $c\in \C$. Assume that 
$(b_1,\ldots,b_k)$ is non-degenerate. 
If $k\geq 2$ or $c\neq 0$, then 
the set
$$Z:=\left\{(m,n)\in\N_0^2:\ P(m)a^m=c+\sum_{i=1}^k Q_i(n)b_i^n\right\}$$ 
is finite.
\end{proposition}
\begin{proof}
When $k\geq 2$, the two linear recurrence sequences
$u_m=P(m)a^m=0\cdot 1^m+P(m)a^m$
and $v_n=c\cdot 1^n+\displaystyle\sum_{i=1}^k Q_i(n)b_i^n$
are not related. Hence $Z$ is finite by part (a) of Theorem~\ref{thm:Schmidt reformulation}. 
If $Z$ is infinite, we have 
$c=0$ by part (b) of Theorem~\ref{thm:Schmidt reformulation}.
\end{proof}

If $p$ is a prime, let $\C_p$ denote the completion of
the algebraic closure of $\Q_p$. It is well-known that $\C_p$ is algebraically closed. We have:
\begin{lemma}\label{lem:value>1}
Let $\gamma\in\C^*$ that is not a root of unity and let
$F$ be a finitely generated subfield of $\C$ containing $\gamma$. Then there exists a field $\cF$ that is either $\C$ or
$\C_p$ together with its usual absolute value $\vert\cdot\vert_0$ and an embedding $\sigma:\ F\rightarrow \cF$
such that $\vert\sigma(\gamma)\vert_0>1$.
\end{lemma}
\begin{proof}
It suffices to prove that there exist $\cF$ that is $\C$ or $\C_p$ and an embedding $\sigma:\ \Q(\gamma)\rightarrow \cF$
satisfying $\vert\sigma(\gamma)\vert_0>1$. Then it is possible
to extend $\sigma$ to $F$ since $\cF$ is algebraically closed and has infinite (in fact, uncountable) transcendence degree over $\Q$.

When $\gamma$ is algebraic, since $\gamma\in\C^*$ is not
a root of unity, a result of Kronecker (see, for instance,
\cite[Theorem~1.5.9]{BG06})
gives that
there is an absolute value  $|\cdot |_v$ of the 
number field $\Q(\gamma)$ such that $|\gamma|_v>1$. This absolute value $|\cdot |_v$ gives rise to the desired embedding 
into $\C$ if $v$ is archimedean, or into $\C_p$ if $|\cdot |_v$
restricts to the $p$-adic absolute value of $\Q$.
When $\gamma$ is transcendental, we simply map $\gamma$
to any transcendental number outside the unit disk.
\end{proof}

\begin{proposition}\label{prop:deg difference}
Let $\alpha,\beta_1,\beta_2\in\C^*$ none of which is
a root of unity. Let $P_1(x)$, $P_2(x)$, $Q_1(x)$,
and $Q_2(x)$ be non-zero polynomials with complex coefficients. Let $Z$ be the set of
$(m,n)\in \N_0^2$ satisfying both $P_1(m)\alpha^m=Q_1(n)\beta_1^n$ and $P_2(m)\alpha^m=Q_2(n)\beta_2^n$. 
If $Z$ is infinite then $\displaystyle\frac{\beta_1}{\beta_2}$ is a root
of unity and $\deg(P_2)-\deg(P_1)=\deg(Q_2)-\deg(Q_1)$.
\end{proposition}
\begin{proof}
For every fixed $m$ (respectively $n$), there are
only finitely many $n$ (respectively $m$) such that 
$(m,n)\in Z$. Hence in every infinite subset of $Z$, $m$ and
$n$ must be unbounded.

Fix any $\epsilon\in \C^*$ such that $P_1(x)+\epsilon P_2(x)$ is not the zero polynomial. 
If $(m,n)\in Z$, we have
$(P_1(m)+\epsilon P_2(m))\alpha^m=Q_1(n)\beta_1^n+\epsilon Q_2(n)\beta_2^n$. By Proposition~\ref{prop:Schmidt note}, we have that $\displaystyle\frac{\beta_1}{\beta_2}$ is a root of unity.

For $(m,n)\in Z$, we have $\vert P_1(m)/P_2(m)\vert=\vert Q_1(n)/Q_2(n)\vert$. By taking $(m,n)\in Z$ when both of $m$ and $n$ are large,
we have that $\deg(P_2)=\deg(P_1)$, $\deg(P_2)>\deg(P_1)$, $\deg(P_2)<\deg(P_1)$
respectively
if and only if $\deg(Q_2)=\deg(Q_1)$, $\deg(Q_2)>\deg(Q_1)$, $\deg(Q_2)<\deg(Q_1)$. 
For the rest of the proof, assume that $\deg(P_2)\neq \deg(P_1)$ and $\deg(Q_2)\neq \deg(Q_1)$. We have $\delta:=\displaystyle \frac{\deg(Q_1)-\deg(Q_2)}{\deg(P_1)-\deg(P_2)}>0$ and we need to prove that
$\delta=1$. 

Assume that $\delta>1$. From $\vert P_1(m)/P_2(m)\vert=\vert Q_1(n)/Q_2(n)\vert$ for $(m,n)\in Z$, we have
$C_1n^{\delta}<m<C_2 n^{\delta}$ for some positive constants
$C_1$ and $C_2$; this is expressed succinctly as 
$m=\Theta(n^{\delta})$.  Let $F$ be the field generated
by $\alpha,\beta_1,\beta_2$, and the coefficients of $P_1,P_2,Q_1,Q_2$. By Lemma~\ref{lem:value>1}, we can embed
$F$ into a field $\cF$ which is $\C$ or $\C_p$ together
with its usual absolute value $\vert\cdot\vert_0$
such that $\vert\alpha\vert_0>1$. Since $m=\Theta(n^{\delta})$
and $\delta>1$,
we have
$\vert Q_1(n)\beta_1^n\vert_0=o(\vert P_1(m)\alpha^m\vert_0)$,
contradiction.

The case $\delta<1$ can be dealt with by similar arguments.
We have $n=\Theta(m^{1/\delta})$. We now embed $F$ into
$\cF$ such that 
$\vert\beta_1\vert_0>1$ to obtain 
$\vert P_1(m)\alpha^m\vert_0=o(\vert Q_1(n)\beta_1^n\vert_0)$,
contradiction. This finishes the proof.
\end{proof}


\section{Proof of Theorem~\ref{main result linear algebra}}
\label{sec:linear algebra}
Since we may restrict to a finitely generated subfield
of $K$ over which all the objects in the statement of
Theorem~\ref{main result linear algebra} are defined, and
we may embed this subfield into $\C$, for the rest
of this section, we assume $K$ is a subfield of $\C$.


\subsection{Some reductions}\label{subsec:reduction linear case}
First, we explain why it suffices to prove Theorem~\ref{main result linear algebra} when $r=2$. Suppose that Theorem~\ref{main result linear algebra} is proved for $r=2$, and assume $r\geq 3$. The
set $S'$ of pairs $(m,n)$ satisfying 
$f_{r-1}^m(\bfp_{r-1})=f_r^n(\bfp_r)$
is a finite union of sets of the form
$\{(t_{r-1}k+\ell_{r-1},t_rk+\ell_r):\ k\in\N_0\}$
for some $t_{r-1},t_{r},\ell_{r-1},\ell_r\in\N_0$. Fix one such set
and the corresponding $t_{r-1},t_r,\ell_{r-1},\ell_r$.
 By ignoring finitely
many pairs $(m,n)$ in $S'$, we may assume that $t_{r-1}$ and $t_r$ are positive. We are now looking for
tuples $(n_1,\ldots,n_{r-2},k)\in\N_0^{r-1}$
such that:
$$f_1^{n_1}(\bfp_1)=\ldots=f_{r-2}^{n_{r-2}}(\bfp_{r-2})=(f_{r-1}^{t_{r-1}})^k(f_{r-1}^{\ell_{r-1}}(\bfp_{r-1})).$$ 
The map $f_{r-1}^{t_{r-1}}$
is associated to the matrix $A_{r-1}^{t_{r-1}}$
whose eigenvalues are not root of unity. So we have reduced to $r-1$ maps 
$f_1,\ldots,f_{r-2},f_{r-1}^{t_{r-1}}$
at the starting points $\bfp_1,\ldots,\bfp_{r-2},f_{r-1}^{t_{r-1}}(\bfp_{r-1})$
that satisfy the hypothesis of Theorem~\ref{main result linear algebra}.

We now focus on the case $r=2$. Let $\tilde{\bfx}_1$ be a fixed point
of $f_1$, equivalently $A_1\tilde{\bfx}_1+\bfx_1=\tilde{\bfx}_1$. 
This is possible since $A_1-I_N$ is invertible. Define $\psi(\bfx)=\bfx+\tilde{\bfx}_1$ so that $\psi^{-1}\circ f_1\circ \psi(\bfx)=A_1\bfx$. Hence $f_1^n(\bfx+\tilde{\bfx}_1)=A_1^n\bfx+\tilde{\bfx}_1$. Similarly, let $\tilde{\bfx}_2$
be a fixed point of $f_2$, then we have
$f_2^n(\bfx+\tilde{\bfx}_2)=A_2^n\bfx+\tilde{\bfx}_2$.
Therefore we reduce to the problem of studying the set of pairs
$(n_1,n_2)\in\N_0^2$ satisfying $A_1^{n_1} \bfu=A_2^{n_2}\bfv+\bfw$
where $\bfu$, $\bfv$, and $\bfw$ are given vectors
such that $\bfu$ (respectively $\bfv$) is not preperiodic
under the map $\bfx\mapsto A_1\bfx$ (respectively 
$\bfx\mapsto A_2\bfx$).

Let $P$ and $Q$ be in $\GL_N(K)$ such that 
$A_1=P^{-1}J_1P$ and $A_2=Q^{-1}J_2Q$ where
$J_1$ and $J_2$ are respectively the Jordan form of $A_1$
and $A_2$. The equation $A_1^{n_1}\bfu=A_2^{n_2}\bfv+\bfw$
is equivalent to 
$J_1^{n_1} P\bfu=PQ^{-1}J_2^{n_2}Q\bfv+P\bfw$. Replacing $(\bfu,\bfv,\bfw)$ by $(P\bfu,Q\bfv,P\bfw)$, we reduce to proving the following
(after a slight change of notation):

\begin{theorem}\label{thm:to Jordan form}
Let $A,B\in M_N(K)$ be in Jordan form and
let $C\in \GL_N(K)$. Let $\bfu,\bfv,\bfw\in K^N$ such that
$\bfu$ and $\bfv$ are respectively not preperiodic 
under the maps $\bfx\mapsto A\bfx$ and $\bfx\mapsto B\bfx$.
If neither $A$ nor $B$ have an eigenvalue which is a root
of unity, then the set
$S:=\{(m,n)\in\N_0^2:\ A^m\bfu=CB^n\bfv+\bfw\}$
is a finite union of sets of the form 
$\{(m_0k+\ell_1,n_0k+\ell_2):\ k\in\N_0\}$
for some $m_0,n_0,\ell_1,\ell_2\in \N_0$.
\end{theorem}

\subsection{Proof of Theorem~\ref{thm:to Jordan form}}
From now on, we assume the notation of Theorem~\ref{thm:to Jordan form}. 
We start with the following easy result:
\begin{lemma}\label{lem:easy ML}
Let $P\in M_N(K)$, $\bfp\in K^N$, and $V$ a closed subvariety
of $\bA^N_K$. The set $\{n\in\N_0:\ P^n\bfv\in V(K)\}$
is a finite union of arithmetic progressions.
\end{lemma}
\begin{proof}
Let $f_1,\ldots,f_k$ be polynomials defining $V$. Then each
$\{f_i(P^n\bfv):\ n\in\N_0\}$ is a linear recurrence sequence and
we can apply Theorem~\ref{thm:SML}. One can also get this result as an immediate consequence of \cite[Theorem~1.3]{Bel06}.
\end{proof}

We may assume that the tuple of non-zero eigenvalues
of $A$ and the tuple of non-zero eigenvalues of $B$ are non-degenerate. This is possible since we can replace the data
$(A,B,C,\bfu,\bfv,\bfw)$ by 
$(A^M,B^M,C,A^{r_1}\bfu,B^{r_2}\bfv,\bfw)$
for some $M\in \N$ and for all $0\leq r_1,r_2\leq M-1$ and
establish the conclusion of Theorem~\ref{thm:to Jordan form}
for the set of pairs $(m,n)\in S$ satisfying
$m\equiv r_1\bmod M$
and $n\equiv r_2\bmod M$.

For $\lambda\in\C$ and $s\in\N$, let $J_{\lambda,s}$ be the Jordan matrix of size $s$ and eigenvalue
$\lambda$. We have the formula:

$$J_{\lambda,s}^n=\begin{bmatrix}
\lambda^n & \binom{n}{1}\lambda^{n-1} & \binom{n}{2}\lambda^{n-2} & \dots &\binom{n}{s-1}\lambda^{n-s+1}\\
0 &  \lambda^n & \binom{n}{1}\lambda^{n-1} & \dots 
& \binom{n}{s-2}\lambda^{n-s+2}\\
\vdots &\vdots &\vdots & \vdots & \vdots\\
\vdots &\vdots &\vdots & \vdots & \vdots\\
0 & 0 & 0 & \dots & \lambda^n
\end{bmatrix}
$$

\emph{For convenience, we follow the convention that assigns
any negative number to be a degree of the zero polynomial.} 

The
key observation is that if $\lambda\ne 0$, then there are polynomials $P_{i,j}$
of degree $j-i$ for $1\leq i,j\leq s$ such that the $(i,j)$-entry
of $J_{\lambda,s}^n$ is $P_{i,j}(n)\lambda^n$ for every $n\in \N$. If $\lambda=0$ \emph{and $n\geq s$}, then $J_{\lambda,s}^n=0_{s,s}$ (the zero matrix) so that we can still express the $(i,j)$-th entry of $J_{\lambda,s}^n$ as $\lambda^n\cdot P_{i,j}(n)$ where $P_{i,j}$ is any chosen polynomial of degree $j-i$.

So, if $\bfa=(a_1,\ldots,a_s)^T$ is a fixed column vector in $K^s$ and $n\geq s$, then there are 
polynomials $P_1,\ldots,P_s$ such that:
$$J_{\lambda,s}^n\bfa=(P_1(n)\lambda^n,\ldots,P_s(n)\lambda^n)^T$$
for every $n\in\N$. Moreover, if $\bfa\neq \mathbf{0}$ and $d:=\max\{j:\ a_j\neq 0\}$
then by a direct calculation, we have $\deg(P_1)=d-1\leq s-1$, 
$\deg(P_2)=d-2,\ldots,\deg(P_s)=d-s\leq 0$.

We assume that the set $S$ is infinite; otherwise there is nothing to prove. Since $\bfu$ and $\bfv$ are not preperiodic,
for every fixed $m$ (respectively $n$), there 
is at most one value of $n$ (respectively $m$) such that
$(m,n)\in S$. Hence $m$ and $n$ must be unbounded in every
infinite subset of $S$. Hence it suffices to prove that the set 
$$S_{\geq N}:=\{(m,n)\in S\colon m,n\ge N\}$$ 
is a finite union of cosets of subsemigroups of $\N_0\times \N_0$ of rank at most equal to $1$. 

Let $p$ be the number of Jordan blocks in $A$, let $J_{\alpha_i,m_i}$ for $1\leq i\leq p$, $\alpha_i\in \C$,
$m_i\in \N$, and $\sum_{i} m_i=N$ 
be the Jordan blocks of
$A$. Let $q$ be the number of Jordan blocks in $A$, let $J_{\beta_j,n_j}$ for $1\leq j\leq q$, $\beta_j\in \C$,
$n_j\in \N$, and $\sum_{j} n_j=N$ 
be the Jordan blocks of
$B$. Note that the $\alpha_i$'s and $\beta_j$'s are not
root of unity.
By a previous observation, there exist polynomials
$P_{i,k}$ for $1\leq i\leq p$ and $1\leq k\leq m_i$
such that for $m\ge N$, we have
\begin{align}\label{eq:formula Amu}
\begin{split}
A^m\bfu=&(P_{1,1}(m)\alpha_1^m,\ldots,P_{1,m_1}(m)\alpha_1^m,P_{2,1}(m)\alpha_2^m,
\ldots,P_{2,m_2}(m)\alpha_2^m,\\
&\ldots,P_{p,1}(m)\alpha_p^m,\ldots,P_{p,m_p}(m)\alpha_p^m)^T.
\end{split}
\end{align}
Moreover, for $1\leq i\leq p$, let $d_i=\deg(P_{i,1})$, then we have
$d_i\leq m_i-1$ and $\deg(P_{i,k})=d_i-k+1$ for $1\leq k\leq m_i$. Similarly, there exist polynomials
$Q_{j,\ell}$ for $1\leq j\leq q$ and $1\leq \ell\leq n_j$
with $e_j:=\deg(Q_{j,1})\leq n_j-1$, $\deg(Q_{j,\ell})=e_j-\ell+1$ such that for $n\ge N$, we have 
\begin{align}\label{eq:formula Bnv}
\begin{split}
B^n\bfv=&(Q_{1,1}(n)\beta_1^n,\ldots,Q_{1,n_1}(n)\beta_1^n,Q_{2,1}(n)\beta_2^n,\ldots,Q_{2,n_2}(n)\beta_2^n,\\
&\ldots,Q_{q,1}(n)\beta_q^n,\ldots,Q_{q,n_q}(n)\beta_q^n)^T.
\end{split}
\end{align}

Since $\bfu$ is not preperiodic under the map $\bfx\mapsto A\bfx$, there is at most one value of $m$ such that $A^m\bfu$ is zero. Hence the set $\cI:=\{i:\ \alpha_i\neq 0\ \text{and\ }d_i\geq 0\}$
is non-empty. Similarly,
the set $\cJ:=\{j:\ \beta_j\neq 0\ \text{and\ }e_j\geq 0\}$ is non-empty. We have the following result.

\begin{proposition}\label{prop:compare 2 rows}
The following hold:
\begin{itemize}
\item [(a)] Let $i\in\{1,\ldots,p\}$ and $k\in\{1,\ldots,m_i\}$
be such that $\alpha_i\neq 0$
and $\deg(P_{i,k})\geq 0$. Then there exist
$j^*\in \{1,\ldots,q\}$,
$\ell^*\in \{1,\ldots,n_{j^*}\}$,
and a polynomial $Q(x)$ such that
$\beta_{j^*}\neq 0$, $\deg(Q)=\deg(Q_{j^*,\ell^*})\geq 0$,
and $P_{i,k}(m)\alpha_i^m=Q(n)\beta_{j^*}^n$
for every $(m,n)\in S_{\geq N}$.
\item [(b)] Let $j\in\{1,\ldots,q\}$ and $\ell\in\{1,\ldots,n_j\}$
be such that $\beta_j\neq 0$
and $\deg(Q_{j,\ell})\geq 0$. Then there exist
$i^*\in \{1,\ldots,p\}$,
$k^*\in \{1,\ldots,m_{i^*}\}$,
and a polynomial $P(x)$ such that
$\alpha_{i^*}\neq 0$, $\deg(P)=\deg(P_{i^*,k^*})\geq 0$,
and $P(m)\alpha_{i^*}^m=Q_{j,\ell}(n)\beta_{j}^n$
for every $(m,n)\in S_{\geq N}$.
\end{itemize}
\end{proposition}
\begin{proof}
For part (a), fix $i\in\{1,\ldots,p\}$ and
$k\in\{1,\ldots,m_i\}$ such that 
$\alpha_i\neq 0$ and $\deg(P_{i,k})\geq 0$. 
Write $\mu=m_1+\ldots+m_{i-1}+1$
so that
$P_{i,1}(m)\alpha_{i}^m$
is the $\mu$-th entry of $A^m\bfu$. Write
$\bfw=(w_1,\ldots,w_N)^T$ and express
the $\mu$-th row of the matrix $C$ as:
$$(c_{1,1},\ldots,c_{1,n_1},\ldots,c_{q,1},\ldots,c_{q,n_q})^T.$$
For $(m,n)\in S_{\geq N}$, from $A^m\bfu=CB^n\bfv+\bfw$, \eqref{eq:formula Amu}, and \eqref{eq:formula Bnv}, we
have:
\begin{equation}\label{eq:compare mu} 
P_{i,1}(m)\alpha_{i}^m - w_{\mu}=
\sum_{j=1}^{q}\left(\sum_{\ell=1}^{n_j} c_{j,\ell}Q_{j,\ell}(n)\right)\beta_j^n=\sum_{j=1}^q Q_j(n)\beta_j^n
\end{equation}
where $Q_j(x):=\displaystyle\sum_{\ell=1}^{n_j}c_{j,\ell}Q_{j,\ell}(x)$.

Recall our assumption that the non-zero elements in $\{\beta_1,\ldots,\beta_q\}$ form a non-degenerate tuple. By
Proposition~\ref{prop:Schmidt note}, $w_{\mu}=0$ 
and there is a unique
$j^*\in\{1,\ldots,q\}$
such that 
$\beta_{j^*}\neq 0$,
$Q(x):=Q_{j^*}(x)\neq 0$,
and for $j\in \{1,\ldots,q\}\setminus\{j^*\}$,
we have $Q_j(n)\beta_j^n\equiv 0$ (this means
either $\beta_j=0$ or $Q_j(x)$ is the zero polynomial)
so that equation \eqref{eq:compare mu} becomes:
\begin{equation}\label{eq:compare mu becomes}
P_{i,1}(m)\alpha_{i}^m=Q(n)\beta_{j^*}^n.
\end{equation}
Since
$Q=\displaystyle\sum_{\ell=1}^{n_{j^*}} c_{j^*,\ell}Q_{j^*,\ell}$,
if we let $\ell^*$
be minimal such that $c_{j^*,\ell^*}\neq 0$ then
$\deg(Q_{j^*,\ell^*})=\deg(Q)\geq 0$. This finishes the 
proof of part (a).

The proof of part (b) is completely similar, this time
we consider the equation
$B^n\bfv=C^{-1}A^m\bfu-C^{-1}\bfw$
and compare the rows corresponding to the entry
$Q_{j,\ell}(n)\beta_j^n$ in $B^n\bfv$. 
\end{proof}

Let $\tilde{i}\in \cI$ be such that 
$d_{\tilde{i}}=\max\{d_i:\ i\in \cI\}$. By part (a) of Proposition~\ref{prop:compare 2 rows}, there exist
$j^*$, $\ell^*$, and a polynomial $Q(x)$ such that 
$e^*:=\deg(Q)=\deg(Q_{j^*,\ell^*})\geq 0$
and
\begin{equation}\label{eq:tilde i and j*}
P_{\tilde{i},1}(m)\alpha_{\tilde{i}}^m=Q(n)\beta_{j^*}^n
\end{equation}
for every $(m,n)\in S_{\geq N}$. We have:
\begin{proposition}\label{prop:dtildei = e*}
The following hold:
\begin{itemize}
\item [(a)] $d_{\tilde{i}}=e^*$, in other words $\deg(P_{\tilde{i},1})=\deg(Q)$.
\item [(b)] There exist $\omega\in K^*$ such that
$\alpha_{\tilde{i}}^m=\omega\beta_{j^*}^n$
for every $(m,n)\in S_{\geq N}$.
\end{itemize}
\end{proposition}
\begin{proof}
First, we prove $d_{\tilde{i}}\leq e^*$ as follows. 
For the entry
$P_{\tilde{i},d_{\tilde{i}}+1}(m)\alpha_{\tilde{i}}^m$
in $A^m\bfu$, we have that 
$\deg(P_{\tilde{i},d_{\tilde{i}}+1})=0$. Applying part (a)
of Proposition~\ref{prop:compare 2 rows} to the
pair $(\tilde{i},d_{\tilde{i}}+1)$, we obtain $j_*$, $\ell_*$,
and a polynomial $R(x)$ such that 
$\beta_{j_*}\neq 0$, $\deg(R(x))=\deg(Q_{j_*,\ell_*})\geq 0$,
and
\begin{equation}\label{eq:tilde i and j_*}
P_{\tilde{i},d_{\tilde{i}}+1}(m)\alpha_{\tilde{i}}^m=R(n)\beta_{j_*}^n
\end{equation}
for every $(m,n)\in S_{\geq N}$. Applying Proposition~\ref{prop:deg difference} to the pair of equations \eqref{eq:tilde i and j*}
and \eqref{eq:tilde i and j_*}, we have that  
$d_{\tilde{i}}=e^*-\deg(R)\leq e^*$. We also have
$\displaystyle\frac{\beta_{j^*}}{\beta_{j_*}}$
is a root of unity, hence $j^*=j_*$ since the tuple of non-zero eigenvalues of $B$ is non-degenerate.

The inequality $e^*\leq d_{\tilde{i}}$ can be proved by similar arguments, as follows. We consider the entry
$Q_{j^*,\ell^*+e^*}(n)\beta_{j^*}^n$ in $B^n\bfv$
for which $\deg(Q_{j^*,\ell^*+e^*})=0$.
Applying part (b) of Proposition~\ref{prop:compare 2 rows}
to the pair $(j^*,\ell^*+e^*)$, we obtain 
$i^*$, $k^*$, and a polynomial $U(x)$ such that 
$\alpha_{i^*}\neq 0$, $\deg(U(x))=\deg(P_{i^*,k^*})\geq 0$,
and 
\begin{equation}\label{eq:i* and j*}
U(m)\alpha_{i^*}^m=Q_{j^*,\ell^*+e^*}(n)\beta_{j^*}^n
\end{equation} 
for every $(m,n)\in S_{\geq N}$. Applying Proposition~\ref{prop:deg difference} to the pair of equations \eqref{eq:tilde i and j*}
and \eqref{eq:i* and j*}, we have that
$d_{\tilde{i}}-\deg(U)=e^*$. This finishes the proof of part (a).

For part (b), from $d_{\tilde{i}}=e^*-\deg(R)$
and part (a), we have that $\deg(R)=0$. Hence both 
polynomials $P_{\tilde{i},d_{\tilde{i}}+1}$
and $R$ are non-zero constants. Since $j_*=j^*$, equation
\eqref{eq:tilde i and j_*} finishes the proof.
\end{proof}

It is possible to use the arguments in Proposition~\ref{prop:dtildei = e*} to prove that 
$d_{\tilde{i}}=\max\{e_j:\ j\in\cJ\}$; however we will not use
this fact. We can now easily complete the proof of Theorem~\ref{thm:to Jordan form}. Part (b) of Proposition~\ref{prop:dtildei = e*} shows that the set
$S_{\geq N}$ \emph{is contained in} one single set of the form
$\{(m_0k+\ell_1,n_0k+\ell_2):\ k\in\N_0\}$
with $\ell_1,\ell_2\in \N_0$ and $m_0,n_0\in \N$. In
fact we can choose $(m_0,n_0)$ to be the minimal pair of positive integers
such that 
$\alpha_{\tilde{i}}^{m_0}=\beta_{j^*}^{n_0}$ and choose $(\ell_1,\ell_2)$
to be
the minimal pair of non-negative integers such that
$\alpha_{\tilde{i}}^{\ell_1}=\omega\beta_{j^*}^{\ell_2}$.
Both of these pairs exist by part (b) of Proposition~\ref{prop:dtildei = e*} and our assumption that $S_{\geq N}$ is infinite.
It remains to study the set of $k\in\N_0$ satisfying:
\begin{equation}\label{eq:set of k}
(A^{m_0})^kA^{\ell_1}\bfu=C(B^{n_0})^kB^{\ell_2}\bfv+\bfw,\ m_0k+\ell_1\geq N,\ n_0k+\ell_2\geq N.
\end{equation}

We now consider $K^{2N}$ with the coordinates $(\bfx,\bfy)$
(where $\bfx,\bfy\in K^{N}$), the linear map
$L:\ K^{2N}\rightarrow K^{2N}$
given by 
$L(\bfx,\bfy)=(A^{m_0}\bfx,B^{n_0}\bfy)$,
the starting point $(A^{\ell_1}\bfu,B^{\ell_2}\bfv)$,
and the subvariety defined by
$\bfx=C\bfy+\bfw$. Applying Corollary~\ref{lem:easy ML} 
to the current data, we have that the set of $k\in\N$ satisfying
\eqref{eq:set of k} is a finite union of arithmetic progressions. This finishes the proof of Theorem~\ref{thm:to Jordan form}.


\section{Proof of Theorem~\ref{main result} and further remarks}
\label{section proof}

\subsection{Some reduction}
By using similar arguments as in Subsection~\ref{subsec:reduction linear case}, we reduce to the case $r=2$. In other words, after
a slight change of notation, we reduce to proving the following:
\begin{theorem}\label{thm:r=2 semiabelian}
Let $X$ be a semiabelian variety over $K$. Let $\Phi$ and
$\Psi$ be $K$-morphisms from $X$ to itself. Let $\Phi_{0}$ 
and $\Psi_0$ be 
 $K$-endomorphisms of $X$ and $\alpha_0,\beta_0\in X(K)$
such that $\Phi(x)=\Phi_0(x)+\alpha_0$
and $\Psi(x)=\Psi_0(x)+\beta_0$. 
Let $\alpha,\beta\in X(K)$ such that 
$\alpha$ is not $\Phi$-preperiodic and $\beta$ is not $\Psi$-preperiodic.
If
none of the eigenvalues of $D\Phi_0$ and $D\Psi_0$ is a root of unity, then the set
$$S:=\{(m,n)\in \N_0\times\N_0\colon \Phi^{m}(\alpha)=\Psi^{n}(\beta)\}$$ 
is a finite union of sets of the form
$\{(m_0k+\ell_1,n_0k+\ell_2)\colon k\in\N_0\}$
for some $m_0,n_0,\ell_1,\ell_2\in\N_0$.
\end{theorem}

\subsection{Proof of Theorem~\ref{thm:r=2 semiabelian}}
Assume the notation of Theorem~\ref{thm:r=2 semiabelian} throughout this subsection. We start with a further reduction.

\begin{lemma}
\label{suffices to multiply by m}
Let $k$ be a positive integer. It suffices to prove Theorem~\ref{thm:r=2 semiabelian} for the maps $\tilde{\Phi}$ and $\tilde{\Psi}$ and starting points $\tilde{\alpha}=k\alpha$ and $\tilde{\beta}=k\beta$, where $\tilde{\Phi}(x)=\Phi_0(x)+k\alpha_0$ and $\tilde{\Psi}(x)=\Psi_0(x)+k\beta_0$.
\end{lemma}

\begin{proof}
Assume Theorem~\ref{thm:r=2 semiabelian} holds for endomorphisms $\tilde{\Phi}$ and $\tilde{\Psi}$ and starting points $\tilde{\alpha}$ and $\tilde{\beta}$. Then $\cO_{\tilde{\Phi}}(\tilde{\alpha})=k\cdot \cO_{\Phi}(\alpha)$ and  $\cO_{\tilde{\Psi}}(\tilde{\beta})=k\cdot \cO_{\Psi}(\beta)$, where for any set $T$ of points of $X$, we define $$k\cdot T :=\{k\cdot x\colon x\in T\}.$$
So, we know that the set $S:=\{(m,n)\in \N_0\times \N_0\colon \tilde{\Phi}^m(\tilde{\alpha}) = \tilde{\Psi}^n(\tilde{\beta})\}$ is a finite union of sets of the form $\{(m_0\ell+r_1,n_0\ell+r_2)\colon \ell\in\N_0\}$ for given $m_0,n_0,r_1,r_2\in\N_0$. So, it suffices to prove that for each such $m_0,n_0,r_1,r_2\in\N_0$, the set of $\ell\in\N_0$ such that $\Phi^{m_0\ell+r_1}(\alpha)=\Psi^{n_0\ell+r_2}(\beta)$ is a finite union of sets of the form $\{\ell_0s+r_0\}_{s\in\N_0}$. Indeed, this last statement is a consequence of \cite[Theorem~4.1]{Jason}.
\end{proof}

Let $R$ be a finitely generated $\Z$-subalgebra of $K$ over which $X$ (along with its points $\alpha, \alpha_0,\beta,\beta_0$), and also $\Phi$ and $\Psi$ are defined.  By \cite[Proposition
4.4]{Jason} (see also \cite[Proposition 3.3]{GT-JNT} and \cite[Chapter~4]{book}), there exist a prime number $p$ and an embedding of $R$ into $\Z_p$ such that
 
\begin{enumerate}
\item[(i)] $X$ has a smooth semiabelian model $\cX$  over $\Z_p$;
\item[(ii)] $\Phi$ and $\Psi$ extend to endomorphism of $\cX$;
\item[(iii)] $\alpha, \alpha_0,\beta,\beta_0$ extend to points in $\cX(\Z_p)$.
\end{enumerate}

Let $f_0$ and $g_0$ denote the linear maps induced on the tangent space at $0$ by $\Phi_0$, respectively $\Psi_0$.  By (ii) above and \cite[Proposition 2.2]{Jason}, if one chooses coordinates
for this tangent space via generators for the completed local ring at
$0$, then the entries of the $N$-by-$N$  matrices $A$ and $B$ corresponding to $f_0$ and $g_0$ will be in $\Z_p$ (where $N=\dim(X)$).  Fix
one such set of coordinates and let $| \cdot |_p$ denote the corresponding $p$-adic
metric on the tangent space at $0$. We let $\C_p$ be the completion of an algebraic closure of $\Q_p$.

According to \cite[Proposition 3, p. 216]{NB} there exists a $p$-adic
analytic map $\exp$ which induces an analytic isomorphism between a
sufficiently small neighborhood $\cU$ of $\C_p^N$ and a
corresponding $p$-adic neighborhood of the origin $0\in
\cX(\C_p)$. Furthermore (at the expense of possibly replacing $\cU$ by a smaller set), we may assume that the neighborhood $\cU$ is
a sufficiently small open ball, i.e., there exists a (sufficiently small) positive real number $\epsilon$ such that $\cU$ consists of all 
$(z_1,\dots,z_N)\in\C_p^N$ satisfying $|z_i|_p<\epsilon$.  Because
$\exp(\cU)\cap \cX(\Z_p)$ is an open subgroup of the compact group
$\cX(\Z_p)$ (see \cite[p.~1402]{GT-JNT}), we conclude that $\exp(\cU)\cap
\cX(\Z_p)$ has finite index $k\in\N$ in $\cX(\Z_p)$. Using Lemma~\ref{suffices to multiply by m} (at the expense of replacing $\alpha$, $\alpha_0$, $\beta$, $\beta_0$ by $k\alpha$, $k\alpha_0$, $k\beta$ and respectively $k\beta_0$), we may assume
that $\alpha,\alpha_0,\beta,\beta_0\in \exp(\cU)$. Therefore, there exists $u,u_0,v,v_0\in
\cU$ such that
$\exp(u)=\alpha$, $\exp(u_0)=\alpha_0$, $\exp(v)=\beta$ and $\exp(v_0)=\beta$. Also, we let $f,g:\bA^N\lra \bA^N$ be the affine transformations given by $f(x)=f_0(x)+u_0$ and respectively $g(x)=g_0(x)+v_0$ for each $x\in \bA^N$. Since the entries of the matrices $A$ and $B$ (corresponding to the linear transformations $f_0$ and $g_0$) have entries which are $p$-adic integers,  we conclude that $\cO_f(u),\cO_g(v)\subset \cU$. So, because $\exp$ is a local isomorphism, while $\Phi^m(\alpha)=\exp\left(f^m(u)\right)$ and $\Psi^n(\beta)=\exp\left(g^n(v)\right)$ for each $m,n\in\N_0$, the desired conclusion follows from Theorem~\ref{main result linear algebra}.

\subsection{Further remarks}
Let $X$ be a semiabelian variety over an algebraically closed field $K$ of characteristic $0$. We conclude this paper by introducing some geometric conditions
that imply the condition that none of
the eigenvalues of  $D\Phi_0$ is a root of unity
as required in the statement of Theorem~\ref{main result} (or 
Theorem~\ref{thm:r=2 semiabelian}).

Recall (see, for example, \cite[Section~7]{Medvedev-Scanlon}) that a dominant $K$-endomorphism $\Phi_0$ of $X$ is said to
preserve a non-constant fibration if
there is a non-constant rational map
$f\in K(X)$ such that 
$f\circ \Phi_0=f$. We have the following:
\begin{proposition}\label{prop:no fibration}
Let $\Phi_0$ be a dominant $K$-endomorphism of $X$. Then $\Phi_0$ preserves a non-constant fibration if and only if at least one of the eigenvalues of $D\Phi_0$ is a root of unity.
\end{proposition}

\begin{proof}
First we note that by \cite[Lemma~4.1]{Ghioca-Scanlon-2}, for any positive integer $\ell$, we know that $\Phi_0$ preserves a non-constant fibration if and only if $\Phi^\ell$ preserves a non-constant fibration.

Assume now that $D\Phi_0$ has an eigenvalue which is a root of unity, say of order $\ell\in\N$. Therefore it suffices to prove that $\Phi_1:=\Phi_0^\ell$ preserves a non-constant fibration.

Let $f\in \Z[z]$ be the minimal (monic) polynomial for $D\Phi_1$; alternatively, $f(z)$ is the minimal monic polynomial with integer coefficients such that $f(\Phi_1)=0\in \End(X)$. We know that $f(1)=0$; hence there exists $g\in \Z[z]$ such that $f(z)=(z-1)\cdot g(z)$. In particular, we know that $g(\Phi_1)$ is not the trivial endomorphism of $X$.  
We let $Y:=g(\Phi_1)(X)$; then $Y$ is a nontrivial semiabelian subvariety of $X$. We also let $\pi:X\lra Y$ be the map $x\mapsto g(\Phi_1)(x)$ and note that $\pi\circ \Phi_1 = \pi$ on $X$. Thus $\Phi_1$ preserves a non-constant fibration, contradiction.

Assume now that $D\Phi_0$ has no eigenvalue which is a root of unity; we will show that $\Phi_0$ does not preserve a nonconstant fibration. Again, at the expense of replacing $\Phi_0$ by an iterate $\Phi_1$, we may assume that all eigenvalues $\lambda_i$ of $D\Phi_1$ have the property that if $\lambda_i/\lambda_j$ is a root of unity, then $\lambda_i=\lambda_j$. Also, note that since $\Phi_0$ and therefore its iterate $\Phi_1$ is a dominant morphism, then each eigenvalue $\lambda_i$ of $D\Phi_1$ is nonzero and not equal to a root of unity according to our hypothesis.

Arguing identically as in the proof of Theorem~\ref{main result}, we can find a prime number $p$ and a suitable model $\cX$ of $X$ over $\Z_p$ such that each entry of $A:=D\Phi_1$ is a $p$-adic integer, and moreover the $p$-adic exponential map $\exp$ induces a local isomorphism between a sufficiently small ball $\cB$ in $\Cp^N$ and a corresponding small $p$-adic neighborhood $\cU$ of the origin of $\cX$. 

We can write $A=B^{-1}JB$, where $B$ is an invertible matrix and $J$ is in Jordan form. Because $B$ is invertible, we can choose $\bfv\in \cB$ such that $B\bfv$ has all its entries nonzero. Then arguing identically as in Section~\ref{sec:linear algebra}, we get that the entries of $J^nB\bfv$ are of the form
$$\left(P_{1,1}(n)\lambda_1^n,\cdots,P_{1,m_1}(n)\lambda_1^n,P_{2,1}(n)\lambda_2^n,\cdots , P_{2,m_2}(n)\lambda_2^n,\cdots \cdots , P_{r,m_r}(n)\lambda_r^n\right)^T,$$
where each $P_{i,1}$ is a nonzero polynomial, and moreover $\deg(P_{i,j})=\deg(P_{i,1})-j+1$ for each $i=1,\dots, r$ and for each $j=1,\dots, m_i$. Then for each $\bfw\in \Cp^N$ and for each $a\in \Cp$, we have that $\left\{\bfw^T\cdot J^nB\bfv + a\right\}_{n\ge 1}$ is a linear recurrence sequence with non-degenerate characteristic roots, unless $\bfw$ is the zero-vector. Therefore, given any proper linear subspace $W\subset \Cp^N$, there are finitely many vectors $A^n\bfv=B^{-1}J^nB\bfv$ contained in the same coset $\bfa + W$ (for any given $\bfa\in\Cp^N$). 

Let now $x=\exp(\bfv)\in \cU$. We claim that $\cO_{\Phi_1}(x)$ is Zariski dense in $X$. Indeed, otherwise there exists a coset $\beta+H$ of a proper algebraic subgroup $H$ of $X$ containing infinitely many points from the orbit $\cO_{\Phi_1}(x)$. This last statement  follows by noting that $\cO_{\Phi_1}(x)$ is contained in a finitely generated subgroup of $X$ (because there exists a monic nonzero polynomial $f\in \Z[z]$ such that $f(\Phi_1)=0\in \End(X)$) and then using the classical Mordell-Lang theorem (see \cite{Voj96}). Because all entries of $D\Phi_1$ are $p$-adic integers, then we know that $\cO_{\Phi_1}(x)\subset \cU$. Since $H$ is a proper algebraic subgroup of $X$, then $\exp^{-1}(H\cap\cU) = H_0\cap \cB$ for some proper linear subgroup $H_0\subset \Cp^N$. But then there are infinitely many vectors $A^n\bfv$ contained in a coset of $H_0$, which is a contradiction.

Now, since $\cO_{\Phi_1}(x)$ is Zariski dense in $X$, we immediately get that $\Phi_1$ and thus $\Phi_0$ cannot preserve a non-constant fibration (see also \cite[Section~7]{Medvedev-Scanlon}). This concludes the proof of Proposition~\ref{prop:no fibration}.    
\end{proof}

Another condition has been mentioned in Remark~\ref{rem:no automorphism}.
\begin{proposition}
\label{prop:no automorphism}
Let $\Phi$ be a self-map of $X$ over $K$ with
$\Phi(x)=\Phi_0(x)+\alpha_0$
where $\Phi_0$ is a $K$-endomorphism and $\alpha_0\in X(K)$.
Assume there does not exist $m\in\N$ and a positive dimensional closed subvariety
$Y$ of $X$ such that $\Phi^m$ restricts to an automorphism on $Y$.
Then none of the eigenvalues of $D\Phi_0$ is a root of unity. 
\end{proposition}

\begin{proof}
We argue by contradiction and therefore, we assume $D\Phi_0$ has an eigenvalue which is a root of unity. At the expense of replacing $\Phi$ by an iterate $\Phi^m$ we may assume that all eigenvalues of  $D\Phi_0$  are either equal to $1$, or they are not roots of unity.  We let (similar to the proof of Proposition~\ref{prop:no fibration}) $f\in \Z[z]$ be the minimal monic polynomial such that $f(\Phi_0)=0$. Then $f(t)=(t-1)^r\cdot f_1(t)$ for some monic polynomial $f_1\in \Z[t]$ such that $f_1(1)\ne 0$, and some $r\in\N$. We let $Y:=f_1(\Phi_0)(X)$ and $Z:=\left(\Phi_0-\id|_X\right)^r(X)$, where for any subvariety $V\subseteq X$, we denote by $\id|_V$ the identity map on $V$. Then $Y$ and $Z$ are semiabelian subvarieties of $X$. As proven in \cite[Lemma~6.1]{Ghioca-Scanlon-2}, we have that $X=Y+Z$ and $Y\cap Z$ is finite. Strictly speaking, \cite[Lemma~6.1]{Ghioca-Scanlon-2} is written for endomorphisms of abelian varieties, but the proof goes verbatim to semiabelian varieties since no property applicable only to abelian varieties (such as the Poincar\'e's Reducibility Theorem---see \cite[Fact~3.2]{Ghioca-Scanlon-2}) is used; essentially, all one uses is that the polynomials $(z-1)^r$ and $f_1(z)$ are coprime. So, $\Phi_0$ restricts to an endomorphism $\tau$ of $Z$ with the property that 
$$\left(\tau -\id|_Z\right):Z\lra Z$$ 
is an isogeny. Note that we allow the possibility that $Z$ is the trivial algebraic subgroup of $X$; in this case, it is still true that $\tau-\id|_Z$ is surjective. We let $\beta_0\in Y$ and $\gamma_0\in Z$ such that $\alpha_0=\beta_0+\gamma_0$. Hence there exists $\gamma_1\in Z$ such that 
\begin{equation}
\label{conjugated gamma}
\Phi_0(\gamma_1)+\gamma_0=\gamma_1. 
\end{equation}
In the case $Z$ is the trivial subgroup, then clearly $\gamma_0=\gamma_1=0$.

We claim that $\Phi$ restricts to an automorphism on the positive dimensional subvariety $V:=\gamma_1+Y$ (note that $r\ge 1$ and thus $\dim(Y)\ge 1$).  First we note that $\Phi_0$ restricts to an automorphism $\sigma$ on $Y$; indeed, $\Phi_0$ induces an endomorphism $\sigma$ of $Y$ by definition, and then since $(\sigma - \id|_Y)^r=0$, we get that $\sigma:Y\lra Y$ is an automorphism. Then for each $y\in Y$ we have that
$$\Phi(y+\gamma_1)=\Phi_0(y)+ \Phi_0(\gamma_1)+\beta_0+\gamma_0= \sigma(y)+\beta_0+\gamma_1.$$
Because $\sigma$ is an automorphism of $Y$, while $\beta_0\in Y$, we conclude that $$y\mapsto \sigma(y)+ \beta_0$$ is an automorphism of $Y$. This concludes the proof of Proposition~\ref{prop:no automorphism}.
\end{proof}



\begin{thebibliography}{GTZ11}
\newcommand{\au}[1]{{#1},}
\newcommand{\ti}[1]{\textit{#1},}
\newcommand{\jo}[1]{{#1}}
\newcommand{\vo}[1]{\textbf{#1}}
\newcommand{\yr}[1]{(#1),}
\newcommand{\page}[1]{#1.}
\newcommand{\ppx}[1]{#1,}
\newcommand{\pps}[1]{#1.}
\newcommand{\bk}[1]{{#1},}
\newcommand{\inbk}[1]{in: {#1}}
\newcommand{\xxx}[1]{{arXiv:#1.}}

\bibitem[AB12]{AB12}
\au{B.~Adamczewski and J.~P.~Bell}
\ti{On the set of zero coefficients of algebraic power series}
\jo{Invent. Math.}
\vo{187}
\yr{2012}
\pps{343--393}

\bibitem[Bel06]{Bel06}
\au{J.~P.~Bell}
\ti{A generalised {S}kolem-{M}ahler-{L}ech theorem for affine varieties}
\jo{J. Lond. Math. Soc. (2)}
\vo{73}
\yr{2006}
\pps{367--379}

\bibitem[BG06]{BG06}
\au{E.~Bombieri and W.~Gubler}
\ti{Heights in diophantine geometry}
New Mathematical Monographs, vol.~4, Cambridge University Press, Cambridge, 2006.

\bibitem[BGT10]{Jason}
\au{J.~P.~Bell, D.~ Ghioca, and T.~J.~ Tucker}
\ti{The dynamical Mordell--Lang problem for \`etale maps}
\jo{Amer. J. Math.}
\vo{132}
\yr{2010}
\pps{1655--1675}

\bibitem[BGT16]{book}
\au{J.~P.~Bell, D.~Ghioca, and T.~J.~Tucker}
\ti{The Dynamical Mordell-Lang Conjecture}
Mathematical Surveys and Monographs \textbf{210}, American Mathematical Society, Providence, RI, 2016, xiv+280 pp.

\bibitem[Bou98]{NB}
\au{N.~Bourbaki}
\bk{Lie Groups and Lie Algebras. Chapters 1--3}
Springer--Verlag, Berlin, 1998.

\bibitem[Den94]{Den94}
\au{L.~Denis} 
\ti{G\'eom\'etrie et suites r\'ecurrentes}
\jo{Bull. Soc. Math. France}
\vo{122}
\yr{1994}
\pps{13--27}

\bibitem[Fal94]{Fal94}
\au{G.~Faltings} 
\ti{The general case of {S.}~{L}ang's conjecture}
Barsotti Symposium in Algebraic Geometry (Albano Terme, 1991),
Perspective in Math., vol.~15, Academic Press, San Diego, 1994, pp.~175--182 

\bibitem[GS]{Ghioca-Scanlon-2}
\au{D.~Ghioca and T.~Scanlon}
\ti{Density of orbits of endomorphisms of abelian varieties}
\jo{Tran. Amer. Math. Soc.}
to appear.

\bibitem[GT09]{GT-JNT}
\au{D.~Ghioca and T.~J.~Tucker} 
\ti{Periodic points, linearizing maps, and the dynamical Mordell-Lang problem}
\jo{J. Number Theory}
\vo{129}
\yr{2009}
\pps{1392--1403}

\bibitem[GTZ08]{Inventiones}
\au{D.~Ghioca, T.~J.~Tucker, and M.~E.~Zieve}
\ti{Intersections of polynomial orbits, and a dynamical Mordell-Lang conjecture}
\jo{Invent. Math.}
\vo{171}
\yr{2008}
\pps{463--483}

\bibitem[GTZ11]{GTZ-Bordeaux}
 \au{D.~Ghioca, T.~J.~Tucker, and M.~E.~Zieve}
\ti{The Mordell-Lang question for endomorphisms of semiabelian varieties}
\jo{J. Theor. Nombres Bordeaux}
\vo{23}
\yr{2011}
\pps{645--666}



\bibitem[GTZ12]{GTZ-Duke}
\au{D.~Ghioca, T.~J.~Tucker, and M.~E.~Zieve} 
\ti{Linear relations between polynomial orbits}
\jo{Duke Math. J.}
\yr{2012}
\vo{161}
\pps{1379--1410}




\bibitem[MS14]{Medvedev-Scanlon}
\au{A.~Medvedev and T.~Scanlon}
\ti{Invariant varieties for polynomial dynamical systems}
\jo{Ann. Math. (2)}
\vo{179}
\yr{2014}
\pps{81-177}

\bibitem[NW14]{Noguchi} 
\au{J.~Noguchi and J.~Winkelmann}
\ti{Nevanlinna theory in several complex variables and Diophantine approximation} 
Grundlehren der Mathematischen Wissenschaften [Fundamental Principles of Mathematical Sciences], \textbf{350}. Springer, Tokyo, 2014. xiv+416 pp.

\bibitem[Sch03]{Sch03}
\au{W.~Schmidt}
\ti{Linear recurrence sequences}
Diophantine Approximation (Cetraro, Italy, 2000),
Lecture Notes in Math. 1819, Springer-Verlag 
Berlin Heidelberg, 2003, pp.~171--247.

\bibitem[SY13]{SY2013}
\au{T.~Scanlon and Y.~Yasufuku}
\ti{Exponential-polynomial equations and dynamical return sets}
\jo{Int. Math. Res. Not. IMRN}
\vo{2014}
\yr{2014}
\pps{4357--4367}

\bibitem[Voj96]{Voj96}
\au{P.~Vojta}
\ti{Integral points on subvarieties of semiabelian varieties, {I}}
\jo{Invent. Math.}
\vo{126}
\yr{1996}
\pps{133--181}

\end{thebibliography}
\end{document}